\documentclass[a4paper,11pt]{article}

\usepackage[a4paper,margin=3.5cm]{geometry}
\usepackage[T1]{fontenc}
\usepackage[utf8]{inputenc}
\usepackage[cm]{aeguill}
\usepackage{amsfonts,amsthm,amssymb,amsmath}
\usepackage{amssymb}
\usepackage{dsfont}
\usepackage{mathrsfs}
\usepackage{graphicx}
\usepackage{multicol}

\newtheorem{Th}{Theorem}
\newtheorem{Lem}{Lemma}
\newtheorem{Def}{Definition}
\newtheorem{Prop}{Proposition}
\newtheorem{Cor}{Corollary}

\title{Some results about ergodicity in shape for a crystal growth model}
\author{F. Ezanno\footnote{\noindent CMI, 
Universit\'e de Provence, 39, rue F. Joliot Curie, 13453 Marseille Cedex 13, FRANCE, \textit{fezanno@cmi.univ-mrs.fr} \newline \textit{AMS} 2010 \textit{subject classification}. Primary, 60J27; secondary 60G55, 60K25, 60J10. \newline \textit{Keywords and phrases.} Markov chain, random deposition, positive recurrence.}}
\date{}

\newcommand{\RR}{\mathbb{R}}
\newcommand{\ZZ}{\mathbb{Z}}
\newcommand{\EE}{\mathbb{E}}
\newcommand{\PP}{\mathbb{P}}
\newcommand{\NN}{\mathbb{N}}

\renewcommand{\H}{\mathcal{H}}

\newcommand{\ind}{\mathds{1}}
\newcommand{\A}{\mathcal E}

\newcommand{\D}{\mathcal D}
\newcommand{\as}{a.\@s.\@}
\newcommand{\ze}{\mathbf{0}}
\newcommand{\ud}{\mathrm{d}}
\newcommand{\eps}{\varepsilon}
\newcommand{\X}{\tilde X}
\renewcommand{\d}{\tilde d}

\newcommand{\zein}{0\slash\infty}
\newcommand{\oexp}[1]{\mathcal{O}_{exp}(#1)}
\newcommand{\oexpt}[2]{\mathcal{O}_{exp}^{#1}(#2)}
\newcommand{\tl}{\frac{t-\ell}{L}}
\newcommand{\PPP}[1]{\PP_0\left(#1\right)}

\begin{document}
\maketitle
\begin{abstract}
We study a crystal growth Markov model proposed by Gates and Westcott (\cite{Kinetics1}, \cite{Kinetics2}). This is an aggregation process where particles are packed in a square lattice accordingly to prescribed deposition rates. This model is parametrized by three values $(\beta_i,~i=0,1,2)$ corresponding to depositions on three different types of sites. The main problem is to determine, for the shape of the crystal, when recurrence and when ergodicity do occur. In \cite{AMS} and \cite{MarkovModels} sufficient conditions are given both for ergodicity and transience. We establish some improved conditions and give a precise description of the asymptotic behavior in a special case. 

\end{abstract}

\section{Definitions and first properties}  

Let $n$ be an integer, $n\geq 2$. We consider a set of $n$ aligned sites, each
site corresponding to a growing pile of particles. The state of a lamellar
crystal (see \cite{Sadler}) is described by a vector $x=(x(1),\dots,
x(n))\in\NN^n$, where the value of $x(i)$ may be thought of as the height of the
pile above site $i$. If $1\leq j\leq n$, $e_j$ will stand for the unitary
vector: $$e_j(i)=\delta_{i,j}.$$ 
For $x\in \NN^n$ and $1\leq j\leq n$, let $V_j(x)$ be the number of sites
adjacent to $j$ whose pile is strictly higher than the pile at site $j$. Namely,
$$V_j(x)=\ind_{\{ x(j-1)>x(j) \}}+\ind_{\{ x(j+1)>x(j) \}}\in\{0,1,2\}.$$ 
For $V_j(x)$ to be well-defined for $j=1$ and $j=n$, we adopt from now on the
convention that $x(0)=x(n+1)=0$, unless otherwise specified. This is the
so-called \emph{zero condition}, which amounts to add a leftmost and a rightmost
site that stay at height $0$ forever. Another natural convention is the
\emph{periodic condition} that consists in deciding that $x(0)=x(n)$ and
$x(n+1)=x(1)$, but we believe that all the results here can be transposed to
periodic condition (in the same way as Theorem 1.1 in \cite{AMS}). We shall also
use the \emph{infinite condition} (resp.\@ the \emph{zero-infinite condition}),
that is $x(0)=x(n+1)=\infty$ (resp.\@ $x(0)=0,x(n+1)=\infty$), and anything
relative to this condition will be denoted with the superscript $\infty$
(resp.\@ the superscript $\zein$).
\begin{Def} Let $n\geq 2$ and $\beta=(\beta_0,\beta_1,\beta_2)
\in]0,+\infty[^3$. We say that $(X^n_t, t\geq 0)$ is a \emph{crystal process}
with $n$ sites and parameter $\beta$ if it is a Markov process on $\NN^n$ with
transition rates given by 
$$\begin{cases}q(x,x+e_j)=\beta_{V_j(x)}, & j=1,\dots,n,\\ q(x,y)=0, & \text{if
} y\notin\{x+e_1,\dots,x+e_n\}.\end{cases}$$
\end{Def}
For a configuration $x$, we define the shape $h$ of $x$ by $$h=\left(\Delta_1
x,\dots,\Delta_{n-1}x\right),$$ where $$\Delta_j x=x(j)-x(j+1),~j=1,\dots,
n-1.$$ Knowing $h$ is equivalent to knowing $x$ up to vertical translation. It
is important to remark that $V_j(x)$ only depends on $x$ through $h$, and
$V_j(h)$ will denote the value of $V_j(x)$ for any $x$ whose shape is $h$. Let
us define, for $j=1,\dots,n$, the vector $$f_j=\begin{cases} e_1,&\text{ if
}j=1,\\ e_j-e_{j-1},&\text{ if }1<j<n, \\ -e_{n-1},&\text{ if
}j=n.\end{cases}$$ 
The object of main interest is the process of the shape of $X^n$, that we now
define, rather than the process $X^n$ itself.

\begin{Def} The \emph{shape process} with $n$ sites and parameter $\beta$ is
defined by $$H^n_t=\left(\Delta_1 X^n_t,\dots,\Delta_{n-1}X^n_t\right),$$ where
$X^n$ is a crystal process with $n$ sites and parameter $\beta$.
$H^n$ is a Markov process on $\ZZ^{n-1}$ with transition mechanism given by 
$$\begin{cases}q(h,h+f_j)=\beta_{V_j(h)}, & j=1,\dots,n,\\ q(h,h')=0, & \text{if } h'\notin\{h+f_1,\dots,h+f_n\}.\end{cases}$$
\end{Def}

These processes have a basic symmetry property, namely the process
$(X^n_t(n),\dots,$ $X^n_t(1))$ has the same distribution as $X^n$, and
consequently the process $(-\Delta_{n-1}X^n,\dots$\\$-\Delta_1 X^n)$ has the
same distribution as $H^n$.
There is a convenient construction of $X^n$, and hence of $H^n$, that we now
describe and will later refer as the \emph{Poisson construction}. As we will see
later, the interest of this construction is to yield useful couplings. Let
$b_0$, $b_1$ and $b_2$ be the $\beta_k$'s ranked in the increasing order. We
take a family of Poisson processes $(N_{k,j},~0\leq k\leq 2,~1\leq j\leq n)$
such that 

\begin{itemize}
\item[-]$N_{k,j}$ has intensity $b_k$,
\item[-] the triples $(N_{0,j},N_{1,j},N_{2,j})_{1\leq j\leq n}$ are mutually independent,
\item[-]for any $j$ there exist three processes $\tilde N_{0,j},~\tilde N_{1,j}$ and $\tilde N_{2,j}$, mutually independent, with intensities $b_0,~b_1-b_0$ and $b_2-b_1$ respectively, such that 
\begin{equation*}
 N_{0,j}=\tilde N_{0,j},
 \quad N_{1,j}=\tilde N_{0,j}+\tilde N_{1,j},
 \quad N_{2,j}=\tilde N_{0,j}+\tilde N_{1,j}+\tilde N_{2,j}.\end{equation*}
\end{itemize}

We build the process $(X^n_t,t\geq 0)$ starting from $x_0$ letting $X^n_0=x_0$,
and at any jump time $t$ of some $N_{k,j}$, 
\begin{equation}X^n_t=\begin{cases}X^n_{t^-}+e_j, &\text{ if
}\beta_{V_j(X^n_{t^-})}\geq b_k, \\ X^n_{t^-}, 
&\text{otherwise.}\end{cases}\label{poisson}
\end{equation}
It is not hard to check that this process has the Markov property and the
desired jump rates. Hence it is a crystal process starting from $x_0$ with $n$
sites and parameter $\beta$.
\\For any positive function $f$ on $\NN$ or $\RR_+$, we write $$f(x)=\oexp{x}$$
if there exists $\alpha,C>0$ such that $f(x)\leq Ce^{-\alpha x}$. If $f$ also
depends on some other variable $t$, the notation $$f(x,t)=\oexpt{t}{x}$$ means
that the same inequality holds with constants $C$ and $\alpha$ being independent
of $t$.
\\We say that the process $X^n$ is \emph{ergodic in shape}, resp.\@
\emph{transient in shape}, whenever the process $H^n$ is ergodic, resp.
transient. The notation $\PP_x$, resp.\@ $\PP_h$, will stand for the
distribution of the trajectory $(X^n_t,t\geq 0)$ starting from $x$, resp. of\@
the trajectory $(H^n_t,t\geq 0)$ starting from $h$. If there is any ambiguity on
the parameter $\beta$, the notation $\PP_x^\beta$ will be used instead. The null
vector will be denoted by $\ze$.
\\This simple model was first described by Gates and Westcott in
\cite{Kinetics1}, where attention was focused on the special case
$\beta_2>\beta_0$, and $$\beta_1=(\beta_0+\beta_2)/2.$$ Under this assumption
the process $H^n$ with periodic conditions enjoys a remarkable \emph{dynamic
reversibility} property that implies ergodicity, and even allows to derive an
exact computation of the invariant distribution. Unfortunately without this
assumption on $\beta$, there is no such simple way to determine whether
ergodicity occurs or not. However we can make a naive remark: since $\beta_0$ is
the statistic speed of peaks and $\beta_2$ is the one of holes, basic intuition
says that increasing $\beta_0$ should make the shape more irregular, making the
process $H^n$ more likely to be transient. Conversely, increasing $\beta_2$
should make the shape smoother, making the process $H^n$ more likely to be
recurrent.\\Gates and Westcott later proved several results about the problem of
recurrence in shape for other parameters, by means of Foster criteria with quite
simple Lyapunov functions. Theorem 2 in \cite{MarkovModels} states that for
periodic conditions and $n\geq 2$, $H^n$ is transient if $\beta_2<\beta_0$.
Ergodicity is shown to hold for
\begin{equation}\beta_1,\beta_2>(n-1)^2\beta_0\label{ncarre},\end{equation} and
a similar condition for ergodicity is also obtained for a process with a
two-dimensional grid of sites. Of course when $n$ is large such conditions are
very restrictive.
\\The family $(\Delta_jX^n_t,t\geq 0)$ is said to be \emph{exponentially tight}
if $$\PP_\ze(|\Delta_j X^n_t|\geq k)=\oexpt{t}{k}.$$ We also say that the family
$(H^n_t,t\geq 0)$ is exponentially tight if for $j=1,\dots n-1$,
$(\Delta_jX^n_t,t\geq 0)$ is exponentially tight. Obviously, exponential
tightness of the process $(H^n_t,t\geq 0)$ implies that it is ergodic with an
invariant distribution having exponential tails.
\\From Theorem 1.2 in \cite{AMS} we get: 

\begin{Th}\label{AMS}If $n\geq 2$ and $\beta_0<\beta_1\leq\beta_2$ then
$(H^n_t,t\geq 0)$ is exponentially tight, and hence ergodic. Moreover there
exists $d_n<\beta_1$ such that $$\PP_\ze(X^n_t(n)\geq d_n t)=\oexp{t}.$$
\end{Th} 

We point out that this result is actually given for $\beta_0<\beta_1<\beta_2$
but the reader may verify that its proof works exactly the same if we take
$\beta_1=\beta_2$. We will pick up several ideas of the approach in \cite{AMS}
in order to give weaker conditions for ergodicity in shape. Our notations will
be consistent with this reference as much as possible. Before stating our
results we begin by defining two useful notions: growth rate and monotonicity.
\begin{Prop}
Suppose that $H^n$ is ergodic and let $\pi^n$ be its invariant distribution. There exists $v^n>0$ such that for $j=1,\dots,n$, almost surely, $$\lim_{t\to\infty}\frac{X^n_t(j)}{t}=v^n.$$
Moreover, for any $j=1,\dots,n$, we have $v^n=\sum_{h\in\ZZ^{n-1}}\beta_{V_j(h)}\pi^n(h)$.
\end{Prop}
\begin{proof}
Let $j\in\{1,\dots,n\}$. Since $(X^n_t(j),t\geq 0)$ is a counting process with
intensity $\beta_{V_j(X^n_t)}$, we have that 
\begin{equation}M_t:=X^n_t(j)-\int_0^t \beta_{V_j(X^n_s)}\ud s\text{ is a
martingale},\label{mart1}\end{equation} and also

\begin{equation}L_t:=M_t^2-\int_0^t \beta_{V_j(X^n_s)}\ud s\text{ is a
martingale}.\label{mart2}\end{equation} Since ergodicity yields the a.s.
convergence $$\lim_{t\to\infty}\frac 1t\int_0^t \beta_{V_j(X^n_s)}\ud
s=\sum_{h\in\ZZ^{n-1}}\beta_{V_j(h)}\pi^n(h),$$ it now remains to show that
$\lim t^{-1}M_t=0$, \as~ But (\ref{mart1}) allows us to use Doob's inequality:
for $r\leq t$,

\begin{align*}
\PP\left(\sup_{s\in[r,t]}\frac{|M_s|}{s}\geq\eps\right) &\leq
\PP\left(\sup_{s\in[r,t]}|M_s|\geq r\eps\right)\\
&\leq \frac{\EE[M_t^2]}{r^2\eps^2}\\
&\leq\frac{Kt}{r^2\eps^2},
\end{align*}
where $K=\max(\beta_0,\beta_1,\beta_2)$. The last inequality is a direct
consequence of (\ref{mart2}). Thus we have $\PP(\sup_{n^2\leq
s\leq(n+1)^2}s^{-1}|M_s|\geq\eps)\leq K(n+1)^2/(\eps^2 n^4)$, and we may can
conclude by Borel-Cantelli's lemma.
\end{proof}

For $n=2$, the simplicity of the dynamics allows us to compute the exact value
of the growth rate:
\begin{Prop}\label{2sites}
$H^2$ is ergodic if and only if $\beta_1>\beta_0$. In this case, $v^2=\frac{2\beta_0\beta_1}{\beta_0+\beta_1}$.
\end{Prop}
\begin{proof}
$H^2_t=X^2_t(1)-X^2_t(2)$ is a random walk on $\ZZ$, whose jump rates are given
by: $$\begin{cases}q(i,i+1)=\beta_0,~q(i,i-1)=\beta_1,&\text{ if } i>0,\\
q(0,1)=q(0,-1)=\beta_0,&   \\  q(i,i+1)=\beta_1,~q(i,i-1)=\beta_0,&\text{ if }
i<0.\end{cases}$$
Thus the first assertion is straightforward. If $\beta_1>\beta_0$, it is easy
to check that the probability measure
$$\mu(\{i\})=\frac{\beta_1-\beta_0}{\beta_1+\beta_0}\left(\frac{\beta_0}{\beta_1
}\right)^{|i|}$$ is a reversible measure for this random walk, so it is the
invariant distribution of the process. We can then compute:
$$v^2=\sum_{i\in\ZZ}\pi_2(\{i\})(\beta_0\ind_{i\geq 0}+\beta_1\ind_{i<
0})=\beta_0\pi^2(\ZZ_+)+\beta_1\pi^2(\ZZ_-^*)=\frac{2\beta_0\beta_1}{
\beta_0+\beta_1}.$$
\end{proof}
We define the canonical partial order $\leq$ in an obvious way: for two
configurations $x,y\in\NN^n$, we write $x\leq y$ if $$x(j)\leq y(j),~1\leq j\leq
n.$$ The process $X^n$ is said to be \emph{attractive} if for any $x\leq y$,
there exists a coupling of two processes $(X^n_t,t\geq 0)$ and $(Y^n_t,t\geq
0)$, with distributions $\PP_x$ and $\PP_y$, such that almost surely, 
$$\forall t\geq 0,~X^n_t\leq Y^n_t.$$
\begin{Lem}\label{mon}
Let $n\geq 2$. If $\beta_0\leq\beta_1\leq\beta_2$, then $X^n$ is attractive.
\end{Lem}
\begin{proof}
Let $x\leq y$. We consider $X^n$ and $Y^n$ obtained with the above Poisson
construction, with $X^n_0=x$, $Y^n_0=y$, both using the same Poisson processes.
Suppose that $$X^n_{t^-}(j)=Y^n_{t^-}(j)$$ and $N_{k,j}$ jumps at time $t$. We
then have $V_j(X^n_{t^-})\leq V_j(Y^n_{t^-})$. Consequently, if $X^n_\cdot(j)$
jumps at time $t$, then so does $Y^n_\cdot(j)$ and hence the inequality
$X^n(j)\leq Y^n(j)$ is preserved.
\end{proof}
We are now interested in comparing two processes with same initial states, but
different numbers of sites, or different parameters. In general it is not true
that increasing one of the parameters increases the process himself. However we
have a weaker result which is sufficient for our purpose.
\begin{Lem}
Let $n\geq 2$ and $\beta\in]0,+\infty[^3$. If $\beta_0\leq\beta_1\leq\beta_2$
and $n\leq m$, then there exists a coupling of two processes $X^n$ and $X^m$
distributed as $\PP_x^{\beta,n}$ and $\PP_x^{\beta,m}$, such that $$\forall
t\geq 0,~X^n_t\leq X^m_t.$$
Let $\beta'\in]0,+\infty[^3$. If $\beta$ and $\beta'$ are such that
$\beta_k\leq\beta'_\ell$ for $k\leq \ell$, then there exists a coupling of two
processes $X^n_t$ and $X'^n_t$ distributed as $\PP_x^{\beta,n}$ and
$\PP_x^{\beta',n}$, such that $$\forall t\geq 0,~X^n_t\leq X'^n_t.$$
\label{lemcouplage}
\end{Lem}
\begin{proof}
Here again we can use Poisson constructions in such a way that the obtained
processes enjoy the desired properties. The details are left to the reader.
\end{proof}

\section{Results}  
As already noticed, $X^n$ is transient in shape (for periodic conditions) when
$\beta_2<\beta_0$. This is not a surprise since this inequality says that peaks
grow faster than holes. Our first Theorem describes more precisely the
asymptotic behaviour of the process $X^n$, with zero-condition, under this
assumption. It says that almost surely the shape ultimately adopts a \emph{comb}
shape. The exact form of the comb actually depends on the position of $\beta_1$
relatively to $\beta_2$ and $\beta_0$ so we actually establish three analogue
results. To illustrate this we show three simulations showing realizations of $t^{-1}X^n_t$ for
$t=1000$ and three different parameters.

\begin{figure}[h]
\begin{center}
\includegraphics[scale=0.16]{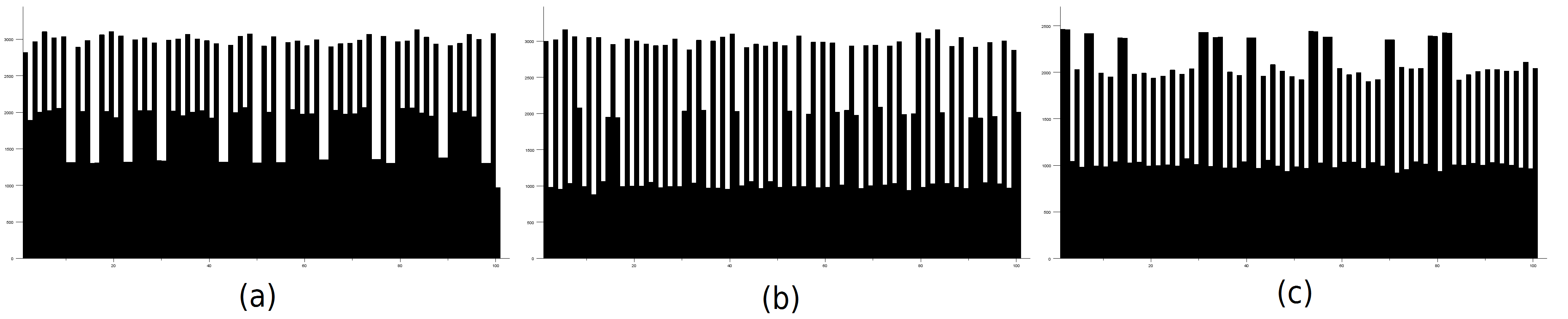}
\caption{$X^n_t$ pour $n=100$,
$t=1000$, et \textbf{(a)} $\beta=(3,1,2)$, \textbf{(b)} $\beta=(3,2,1)$, \textbf{(c)} $\beta=(2,3,1)$.}
\label{peignes}
\end{center}
\end{figure}

Before stating the result we need to introduce some further notation. We write
$t^{-1}X^n_t$ for the vector $\left(t^{-1}X^n_t(1),\dots,t^{-1}X^n_t(n)\right)$.
For two vectors $a=(a_1,\dots a_k)$ and $b=(b_1,\dots b_\ell)$ we denote by
$(a,b)$ the vector $(a_1,\dots a_k,b_1,\dots b_\ell)$. $(-)$ denotes the empty
vector. Let  $\A_1$ be the set of all the $n$-uples of the form
\begin{equation}
(a_1,\beta_2,a_2,\dots,a_{k-1},\beta_2,a_k),\label{A1}\end{equation}
where $k\in\NN$, and $a_i=(\beta_0)\text{ or }(v^2,v^2)$ for $1\leq i\leq k$.
Similarly we define $\A_2$ as the set of all the $n$-uples of the form
\begin{equation} (\beta_{i_1},\dots \beta_{i_n}),\label{A2}\end{equation}
where $i_j\in\{0,1,2\}$,~$i_j\neq i_{j+1} \text{ for }1\leq j\leq n-1$ and
$i_1,i_n\neq 2$. Let $H^{n,\infty}$ be the shape process with infinite
condition. The proof of Proposition \ref{2sites} also works with $H^2$ being
replaced by $H^{2,\infty}$, $\beta_0$ by $\beta_1$ and $\beta_1$ by $\beta_2$.
Thus, whenever $\beta_2>\beta_1$ the process $H^{2,\infty}$ is ergodic with
growth rate $v^{2,\infty}=2\beta_1\beta_2/(\beta_1+\beta_2)$. Let $\A_3$ be the
set of all $n$-uples of the form
\begin{equation}
(e_L,\beta_0,a_1,\beta_0,a_2,\dots,a_{k-1},\beta_0,a_k,\beta_0,e_R),\label{A3}
\end{equation}
where $k\in\NN$, $a_i=(\beta_2)\text{ or }(v^{2,\infty},v^{2,\infty})$ for
$1\leq i\leq k$, and $e_L,e_R=(\beta_1)$ or $(-)$.
\\ In Section \ref{preuvepeigne} we prove:
\begin{Th}\label{peigne}~Let $n\geq 2$ and $x\in\ZZ^n$.
\begin{itemize}
\item[(i)] If $\beta_2<\beta_0\leq\beta_1$ then $t^{-1}X^n_t$ converges
$\PP_x$-a.\@s.\@ and
$$\PP_x\left(\lim_{t\to\infty}t^{-1}X^n_t\in\A_1\right)=1.$$
\item[(ii)]If $\beta_2<\beta_1<\beta_0$ then $t^{-1}X^n_t$ converges
$\PP_x$-a.\@s.\@ and
$$\PP_x\left(\lim_{t\to\infty}t^{-1}X^n_t\in\A_2\right)=1.$$
\item[(iii)]If $\beta_1\leq\beta_2<\beta_0$ then $t^{-1}X^n_t$ converges
$\PP_x$-a.\@s.\@ and
$$\PP_x\left(\lim_{t\to\infty}t^{-1}X^n_t\in\A_3\right)=1.$$
\end{itemize}
\end{Th}
\textbf{Remark.} It is plausible that the almost sure convergence of
$t^{-1}X^n_t$ holds even without the assumption $\beta_2<\beta_0$. For instance
when $\beta_0<\beta_2<\beta_1$ our belief, confirmed by computer simulations, is
that it is always the case that the $n$ sites ultimately divide in a certain
number of blocks (possibly one in the ergodic case) of various widths separated
by holes of unit length, each of these blocks being ergodic in shape. If this is
true then each site admits an asymptotic speed which is either $v^k$, $k$ being
the width of the block containing the site, or $\beta_2$ if the site is
ultimately a hole. Unfortunately we have not been able to prove this.
\\The next results concern the process with parameters lying in the domain
$$\D=\{\beta=(\beta_0,\beta_1,\beta_2):\beta_0<\beta_2<\beta_1\}.$$ We point out
that the three degrees of freedom actually reduce to two. We can indeed assume
$\beta_0=1$ because otherwise we can work with the process
$(X^n_{t/\beta_0},t\geq 0)$.
\\ Our first result in that direction is an abstract condition for ergodicity.
The value of $\beta_0$ being fixed from now on, our strategy is to give for each
$\beta_1>\beta_0$ a threshold value of $\beta_2$ above which ergodicity holds.
The main idea is to compare $X^n$ with an auxiliary process $\X^n$ which is
defined as the crystal process with parameters
\begin{equation}\tilde\beta_0=\beta_0,~\tilde\beta_1=\beta_1\text{ and
}\tilde\beta_2=\beta_1.\label{reftilde}\end{equation}
Anything relative to the process $\X^n$ will be denoted with the symbol $\sim$.
For $\beta_1>\beta_0$ and $n\geq 2$ we define 
$$\d_n(\beta_1)=\inf\{ d>0:~\PP_\ze(\X^n_t(n)\geq dt)=\oexp{t} \}.$$
Clearly $\d_n(\beta_1)\in[\beta_0,\beta_1]$. Moreover it follows from Lemma
\ref{lemcouplage} that 

\begin{equation}\d_n(\beta_1)\text{ is an increasing function of both } n\text{
and }\beta_1.\label{dtildencroissant}\end{equation}

\begin{Th}\label{recurrence0}
If $\beta_1>\beta_0$ and $\beta_2>\d_n(\beta_1)$ then $H^k$ is ergodic for
$k\leq n+2$.
\end{Th}

\begin{Cor}\label{n4}
If $\beta_1,\beta_2>\beta_0$ then $H^3$ is ergodic.
\end{Cor}
In section \ref{preuverecurrence} we shall prove Theorem \ref{recurrence0}, and
Theorem \ref{recurrence1} below, which is an application of Theorem
\ref{recurrence0}.

\begin{Th}\label{recurrence1}
Let $n\geq 2$ and $\beta\in\D$. The process $H^k$ is ergodic for $2\leq k\leq
n+2$ if $\beta$ satisfies one of the following conditions:
\begin{itemize}
\item[(a)]$\beta_2>n\beta_0$,
\item[(b)]$\beta_2>((n-1)\beta_1+\beta_0)/n$.
\end{itemize}
Moreover $H^k$ is ergodic for any $k\geq 2$ if $\beta$ satisfies:
\begin{itemize}
\item[(c)]$\beta_2>4\sqrt2\sqrt{\beta_1\beta_0}$. 
\end{itemize}
\end{Th}

Finally in Section \ref{transience} we establish that $H^n$ is transient for
some parameters in $\D$. More precisely we show

\begin{Th}\label{transience}Let $\beta_0>0$ and $\beta_2\in]\beta_0,2\beta_0[$.
Then there exists $B>\beta_0$ such that $H^n$ is transient for any $\beta_1> B$
and $n\geq 5$.
\end{Th}

Before turning to the proofs we briefly comment the interest of the above
assertions, with the following diagram in mind. In Theorem \ref{recurrence1},
condition (a) improves the only sufficient condition for ergodicity in $\D$
established so far, namely (\ref{ncarre}). Condition (b) provides for fixed $n$
a right-side neigbourhood of the set $\{\beta:\beta_0<\beta_1\leq\beta_2\}$ in
which ergodicity still holds. Condition (c) is certainly the most important one
since it yields a zone of ergodicity that does not depend on the number of
sites, and Theorem \ref{transience} does the same for transience.

\begin{figure}[h]
\begin{center}
\includegraphics[scale=0.4]{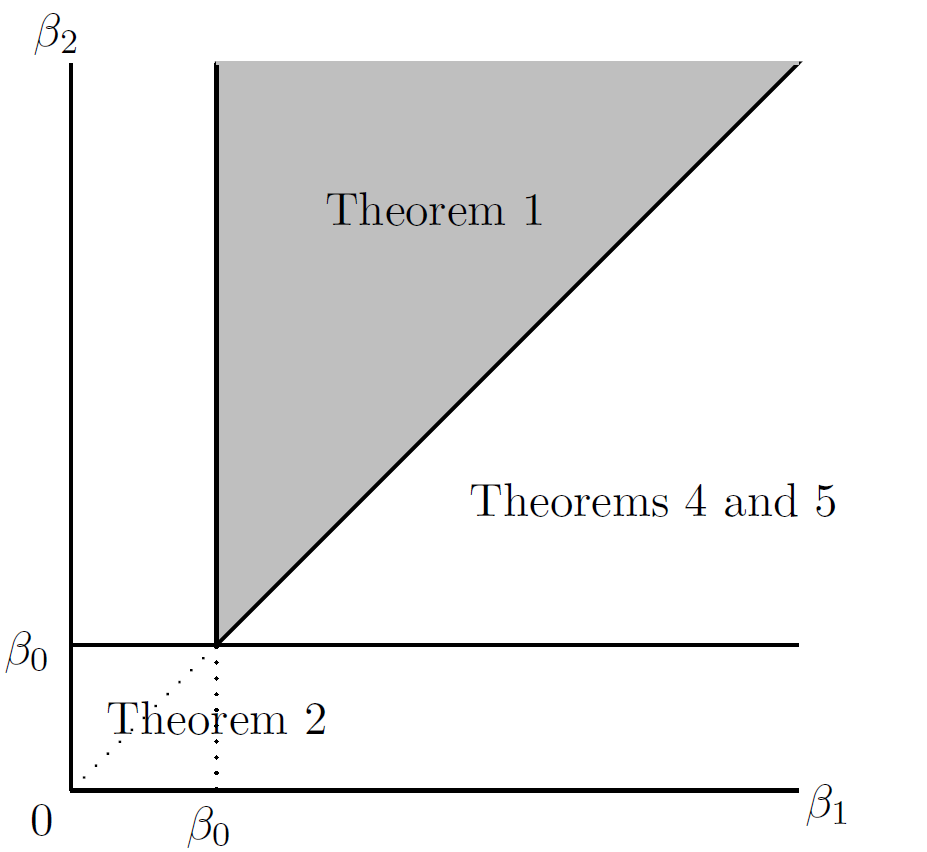}
\end{center}
\end{figure}  

\newpage
 \section{Proof of Theorem \ref{peigne}} \label{preuvepeigne}
The proof of Theorem \ref{peigne} uses the following technical result. 
\begin{Lem}
\label{lemmemarkov}
Let $(Z_t,t\geq 0)$ be a Markov process on some countable set $E$, and
$A\subset E$. For any $F\subset E$ we define $T_F=\inf\{t\geq 0:Z_t\in F\}$. We
assume that there exists $p>0,~N\in\NN$ and some subsets $B_1,\dots ,B_N$ and
$C_1,\dots ,C_N$ of $E$ such that
\begin{itemize} 
\item[\emph{(a)}] for any $x\notin (A\cup B)$, where $B=\cup_{i=1}^N B_i$, we
have $$\PP_x(T_B<+\infty)=1,$$ 
\item[\emph{(b)}] for any $i\in\{1,\dots,n\}$, and $x\in B_i\backslash A$,
$$\PP_x(T_{C_i\cup A}<+\infty)=1,$$ and $$\PP_x(Y_{T_{C_i\cup A}}\in A)\geq p.$$
\end{itemize}
Then for any $x\in A^c$, we have $\PP_x(T_A<+\infty)=1$.
\end{Lem}

\begin{proof}
We consider the partition $(B'_1,\dots B'_N)$ of the set $B$ given by
$B'_1=B_1$ and $$B'_k=B_k\backslash\left(\cup_{i=1}^{k-1}B_i\right),\quad 2\leq
k\leq N.$$ We start off by defining inductively an increasing sequence of
stopping times. For the well-definedness of this sequence we add an element
$\partial$ to the set $E$ and use the conventions $\inf\varnothing=\infty$ and
$Z_\infty=\partial$.
\\ Let $x\notin A$. We define
$$\tau_1:=\inf\{t\geq 0: Z_t\in B\},$$ 
and let $i_1$ be such that $Z_{\tau_1}\in B'_{i_1},$ and
$$\tau'_1:=\begin{cases}\inf\{t\geq \tau_1: Z_t\in C_{i_1}\cup A\},& \text{if }
Z_{\tau_1}\notin A,\\ \infty, &\text{ otherwise}.\end{cases}$$
For $n\geq 2$, we define
$$\tau_n:=\begin{cases}\inf\{t\geq \tau'_{n-1}: Z_t\in B\},   &    \text{if }
Z_{\tau'_{n-1}}\notin A\cup\{\partial\},\\ \infty,     &     
\text{otherwise},\end{cases},$$
and let $i_n$ be such that $Z_{\tau_n}\in B'_{i_n}$ if $\tau_n<\infty$, and
$$\tau'_n:=\begin{cases}\inf\{t\geq \tau_n: Z_t\in C_{i_n}\cup A\},    &   
\text{if } Z_{\tau_{n}}\notin A\cup\{\partial\},\\ \infty, & \text{otherwise}.
\end{cases}$$
In this construction the sequence
$(Z_{\tau_1},Z_{\tau'_1},\dots,Z_{\tau_n},Z_{\tau'_n},\dots)$ is such that
$Z_{\tau_1}\notin A,~Z_{\tau'_1}\notin A,~Z_{\tau_2}\notin A\dots$ until one of
its terms belongs to $A$, and all the following terms are equal to $\partial$.
Proceeding by induction, the strong Markov property and assumptions (a) and (b)
easily yield $$\forall n\geq 1,~ \PP_x( Z_{\tau_1}\notin A, Z_{\tau'_1}\notin
A,\dots Z _{\tau_n}\notin A,Z_{\tau'_n}\notin A)\leq (1-p)^n.$$
Letting $n$ go to infinity in this inequality, we get $\PP_x( \forall t\geq
0,Z_t\notin A)=0$.
\end{proof}

\begin{proof}[Proof of Theorem \ref{peigne} ] For any configuration $h$ and
$i<j$, we denote by $\Delta_{i,j}(h)=h(i)+\dots+h(j-1)$ the height difference
between sites $i$ and $j$. We consider the following subsets of $\ZZ^{n-1}$. To
simplify notations, inside braces we denote by $x$ any configuration whose shape
is $h$:

\begin{itemize}
\item[-]$B_i=\{h\in\ZZ^{n-1}:h(i)=0\},~1\leq i\leq n-1,$
\item[-]$A_i=\{h\in\ZZ^{n-1}:h(i-1)\geq 0,h(i)\leq 0\},~2\leq i\leq n-1,$
\item[-]$A=\cup_{i=2}^{n-1}A_i,$
\item[-]$C_1=\{h\in\ZZ^{n-1}:\min(x(1),x(2))=x(3)\},$
\item[-]$C_i=\{h\in\ZZ^{n-1}:\min(x(i),x(i+1))=\max(x(i-1),x(i+2))\},1\leq i\leq n-2,$
\item[-]$C_{n-1}=\{h\in\ZZ^{n-1}:\min(x(n-1),x(n))=x(n-2)\}.$
\end{itemize}

$C_i$ is the set of configurations in which the lower site of the block
$\{i,i+1\}$ is at the same level as the higher site among the sites neighboring
this block (there are two such sites unless $i=1$ or $i=n-1$). A little moment
of thought will convince the reader of the following fact: in any configuration
$h\in(A\cup B)^c$, there must be a unique site with maximal height. This will be
used several times in this Section.\\Before turning to the proof, we introduce
further notations. If $1\leq a\leq b\leq n$ and $x_0\in\NN^{b-a+1}$ we denote by
\begin{equation}(X^{a:b,x_0,t_0}_t,t\geq 0)\label{defxab}\end{equation} the
crystal process with $b-a+1$ sites starting from $x_0$ and defined in the same
way as in (\ref{poisson}) but using the Poisson processes
$N_{k,j}(t_0+\cdot),~a\leq j\leq b,~0\leq k\leq 2.$ When $t_0=0$ this
superscript will be dropped. Moreover for any vector $x\in\RR^n$, we let
$x(a:b):=(x(a),x(a+1),\dots,x(b))$. 
\\We begin with the proof of (i), proceeding by induction on $n$. The case $n=1$
is straightforward and for $n=2$ the result is a consequence of Proposition
\ref{2sites}. We now take $n\geq 3$ and assume that (i) holds for any $k<n$. For
$Y\subset\ZZ^{n-1}$ we define $T_Y=\inf\{t\geq 0: H^n_t\in Y\}$, and we also use
the following notations:
$$E_Y(t)=\{\forall s\geq t,H^n_s\in Y\},$$
and
$$F_Y=\cup_{t\geq 0}E_Y(t).$$ 
Let $i\in\{2,\dots,n-1\}$. On the event $E_{A_i}(t)$, after time $t$ the value
of $X^n_t(i)$ is increased by one unit at the jump times of $N_{2,i}$, and only
at these times. Indeed, for $s\leq t$, if both $H^n_s(i-1)>0$ and $H^n_s(i)<0$
then site $V_i(H^n_s)=2$, and if one of them is $0$ then any jump of site $i$ is
forbidden by the event $E_{A_i}(t)$. Consequently we have
$E_{A_i}(t)\subset\{\forall s\geq t, X^n_s(i)= X^n_t(i)+N_{2,i}(s)-N_{2,i}(t)\}$
so for any $x\in\ZZ^n$,
\begin{equation}\PP_x\text{-ps, }F_{A_i}\subset\{\lim
t^{-1}X^n_t(i)=\beta_2\}.\label{sitei}\end{equation}
We now use the fact that as long as  $H^n_t\in A_i$, the vectors $X^n_t(1:i-1)$
and $X^n_t(i+1:n)$ evolve like two independent crystal processes with $i-1$
(resp.\@ $n-i$) sites. Namely on the event $E_{A_i}(t_0)\cap\{X^n_{t_0}=x_0\}$,
we have
\begin{itemize}
\item[-] $X^n_{t_0+t}(1:i-1)=X_t^{1:i-1,x_0^\ell,t_0}$, where $x_0^\ell=x_0(1:i-1)$, and
\item[-] $X^n_{t_0+t}(i+1:n)=X_t^{i+1:n,x_0^r,t_0}$, where $x_0^r=x_0(i+1:n)$.
\end{itemize}
Thus the inductive hypothesis ensures that $\PP_x$-\as,
\begin{align}F_{A_i}\subset\big\{ t^{-1}X^n_{t}(1:i-1)\text{ and }
t^{-1}X^n_{t}(i+1:n)&\text{ both converge} \nonumber\\ \text{and their limits
have the form (\ref{A1})}\big\}.&\label{sitejusquei}\end{align}
Thanks to (\ref{sitei}) and (\ref{sitejusquei}) it is sufficent to show that 
\begin{equation}\PP_h(\cup_{i=2}^{n-1}F_{A_i})=1 \label{amontrer}\end{equation} 
to achieve the proof. We first prove the existence of a constant $r>0$ such
that for any $h\in A_i$, 
\begin{equation}\PP_h(E_{A_i}(0))\geq r.\label{troustable}\end{equation}
On one hand, starting from $h\in A_i$ two transitions suffice to make site $i$
strictly lower than its two neighbours, so there exists $r_1>0$ such that for
any $h\in A_i$, 
\begin{equation}\PP_h\big(\forall t\leq 1,~H^n_t\in A_i;~H^n_1(i-1)>0\text{ and
}H^n_1(i)<0\big)\geq r_1.\label{troustable1}\end{equation} 
On the other hand if $h'$ is such that $h'(i-1)>0$ and $h'(i)<0$, we have the
$\PP_{h'}$-a.\@s.\@ inclusion $\{\forall t\geq 0,~N_{2,i}(t)\leq
\min(N_{0,i-1}(t),N_{0,i+1}(t))\}\subset E_{A_i}(0)$. Since $\beta_2<\beta_0$
basic considerations about Poisson processes give the existence of $r_2>0$ such
that for any $h'$ as above,
\begin{equation}\PP_{h'}(E_{A_i}(0))\geq r_2.\label{troustable2}\end{equation} 
Hence (\ref{troustable}) with $r=r_1r_2$ follows from (\ref{troustable1}) and
(\ref{troustable2}). Finally (\ref{amontrer}) will follow form
(\ref{troustable}), the strong Markov property  and the fact that for any
$h\notin A$, \begin{equation}\PP_h(T_A<+\infty)=1,\label{hit}\end{equation}
which now remains to be shown. To show (\ref{hit}) we shall check that
assumptions (a) and (b) of Lemma \ref{lemmemarkov} are fulfilled.
\\For (a) we take $h\notin(A\cup B)$. Let $i$ be the unique site with maximal
height in configuration $h$. If $i>2$ then $h(i-2),h(i-1)<0$ (and if $i=2$ it is
still the case by convention), so that on the event $\{T_B>t\}$, we have
$H^n_t(i-1)=h(i-1)+N_{1,i-1}(t)-N_{0,i}(t),~\PP_h$-a.\@s. Thus
$\PP_h(T_B=+\infty)\leq\PP_h(\forall t\geq
0,h(i-1)+N_{1,i-1}(t)-N_{0,i}(t)<0)=0$. By the symmetry of the process, the case
$i=1$ may be treated as the case $i=n$.
\\We now turn to (b) so we take an initial condition $h\in B_i\backslash A$. It
is easy to see that there is some $i\in\{2,\dots n-1\}$ such that 
\begin{equation}h(1)<0,\dots,h(i-1)<0,h(i)=0,h(i+1)>0,\dots,h(n-1)>0.\label{
forme}\end{equation}
Note that on the event $\{T_{C_i\cup A}>t\}$ all strict inequalities in
(\ref{forme}) have to be preserved up to time $t$, hence $\{T_{C_i\cup
A}>t\}\subset\{\forall s\in[0,t],V_{i-1}(H^n_s)=1\}$. Consequently we have, for
any configuration $x$ whose shape is $h$: 
$$\PP_x\text{-\as},~\{ T_{C_i\cup A}>t \}\subset\{
X^n_t(i-1)=x(i-1)+N_{1,i-1}(t)\}.$$
From a similar argument we also get:
$$\PP_x\text{-\as},~\{ T_{C_i\cup A}>t \}\subset\{
\big(X^n_t(i),X^n_t(i+1)\big)=X_t^{i:i+1,(0,0)}\},$$
and combining the two last inclusions gives $\{T_{C_i\cup
A}>t\}\subset\{H^n_t(i-1)=h(i-1)+N_{1,i-1}(t)-X_t^{i:i+1,(0,0)}(1)\},
~\PP_x\text {-\as}$ Letting $t\to\infty$ then gives
\begin{equation}\PP_h(T_{C_i\cup A}=+\infty)\leq\PP(\forall t\geq
0,h(i-1)+N_{1,i-1}(t)-X_t^{i:i+1,(0,0)}(1)< 0).\label{blip}\end{equation} 
If $\beta_1>\beta_0$ the probability in (\ref{blip}) is equal to $0$ since \as,
$$\lim\frac1t \big( h(i-1)+N_{1,i-1}(t)-X_t^{i:i+1,(0,0)}(1)
\big)=\beta_1-v^2>0,$$
where the last inequality is an easy consequence of the definition of $v^2$. If
$\beta_1=\beta_0$ this probability is still null since
$h(i-1)+X_t^{i:i+1,(0,0)}(1)-N_{1,i-1}(t)$ is then a symmetric random walk on
$\ZZ$. Now by symmetry the case $i=1$ may be treated like the case $i=n-1$. 
\\Finally, the distribution of the process $(X^2_t,t\geq 0)$ being
exchangeable, we have $\PP_h(H^n_{T_{C_i\cup A}}\in C_i\cap
A)\geq\PP_h(H^n_{T_{C_i\cup A}}\in C_i\cap A^c)$. In particular,
$\PP_h(H^n_{T_{C_i\cup A}}\in A)\geq 1/2$ and this concludes the proof of (i).
\\The proofs of (ii) and (iii) are based on the same ideas. Let us continue
with (ii). It is straightforward for $n=1$. For $n=2$ it is easy: the process
$H^2_t\in\ZZ$ is a nearest-neighbour random walk, namely $|H^2_t|$ is increased
by one unit at rate $\beta_0$ and decreased by one unit at rate $\beta_1$
(except of course at $0$). We then have $\PP_h(F_{\NN}\cup F_{-\NN})=1$, and
this allows us to conclude.
\\ As we did for (i), we shall use Lemma \ref{lemmemarkov} to show that
$\PP_h(T_A<+\infty)=1$ for $h\notin A$, and conclude by induction. This time
however, this is true only for $n\geq 4$, so we first have to treat the case
$n=3$ separately.
\\For $n=3$ we let $K_1=\{h\in\ZZ^2: h(1)\geq 0,h(2)\leq0\}$ and
$K_2=\{h\in\ZZ^2: h(1)\leq 0,h(2)\geq 0\}$. As in the proof of (i) there exists
$r>0$ such that $\PP_h(E_{K_i}(0))\geq r$ for any $h\in K_i$, and once we know
that $H^3_t$ stays forever in one of these two sets, we are done. Putting
$K=K_1\cup K_2$ it is again sufficient, thanks to the Markov property, to show
that for $h\in K^c$ we have $\PP_h(T_<+\infty)=1$. Take for exemple $h(1),h(2)>
0$. On the event $\{T_K>t\}$ we have $H^3_t(1)=h(1)+N_{1,1}(t)-N_{1,2}(t)$,
$\PP_h$-\as, so $\PP_h(T_K=+\infty)\leq\PP(\forall t\geq
0,h(1)+N_{1,1}(t)-N_{1,2}(t)> 0)=0$ because of the recurrence of the symmetric
random walk on $\ZZ$.
\\We now fix $n\geq 4$ and check (a) in Lemma \ref{lemmemarkov}. Let $h\notin
(A\cup B)$ and $i$ be the unique site with maximal height in configuration $h$.
We may suppose that $i\geq 3$ without loss of generality thanks to the symmetry.
On the event $\{T_B>t\}$, we have 
$H^n_t(i-2)=h(i-2)+N_{1,i-2}(t)-N_{1,i-1}(t),~\PP_h\text{-\as}$ Again we get
$\PP_h(T_B=+\infty)=0$ by the recurrence of the symmetric random walk.
\\Now we check (b) in Lemma \ref{lemmemarkov}. Let $h\in B_i\backslash A$. We
suppose that $2\leq i\leq n-1$, since the case $i=1$ is the same as $i=n-1$
thanks to the symmetry. On the event $\{T_{C_i\cup A}>t\}$, we have $\PP_h$-\as,
$$\big(H^n_t(i-1),\Delta_{i-1,i+1}H^n_t\big)=\big(h(i-1),\Delta_{i-1,i+1}
h\big)+\big(N_{1,i-1}(t),N_{1,i-1}(t)\big)-X_t^{i:i+1,(0,0)}.$$
We define the events $$G_1(t)=\{\forall s\geq
t,X_s^{i:i+1,(0,0)}(1)<X_s^{i:i+1,(0,0)}(2)\},$$   $$G_2(t)=\{\forall s\geq
t,X_s^{i:i+1,(0,0)}(1)>X_s^{i:i+1,(0,0)}(2)\}.$$
We have $\{T_{C_i\cup A}=+\infty\}\cap G_1(t)\subset\{\forall s\geq
t,H^n_s(i-1)=H^n_t(i-1)+\big( N_{1,i-1}(s)-N_{1,i-1}(t) \big)-
\big(N_{1,i}(s)-N_{1,i}(t) \big)\}$. Thus using the recurrence of the symmetric
random walk we get 
 \begin{align}\nonumber\PP_h(\{T_{C_i\cup A}=+\infty\}\cap
G_1(t))&\leq\PP_h\big(\forall s\geq t,H^n_t(i-1)+(N_{1,i-1}(s)-N_{1,i-1}(t) ) \\
&\quad - (N_{1,i}(s)-N_{1,i}(t))<0\big)=0.\label{G1}\end{align}
Similarly we have $\{T_{C_i\cup A}=+\infty\}\cap G_2(t)\subset\{\forall s\geq
t,\Delta_{i-1,i+1}H^n_s=\Delta_{i-1,i+1}H^n_t+\big( N_{1,i-1}(s)-N_{1,i-1}(t)
\big)-\big(N_{1,i+1}(s)-N_{1,i+1}(t)\big)\}$ and we deduce that 
 \begin{equation}\PP_h(\{T_{C_i\cup A}=+\infty\}\cap
G_2(t))=0.\label{G2}\end{equation}
From the above remark on the crystal process with $2$ sites, we obtain
$$\PP\left[ \left(\bigcup_{t\geq 0}G_1(t)\right)\cup\left(\bigcup_{t\geq
0}G_2(t)\right)\right]=1,$$ and consequently (\ref{G1}) and (\ref{G2}) imply
that $\PP_h(\{T_{C_i\cup A}=+\infty\})=0$. The fact that $\PP_h(H^n_{T_{C_i\cup
A}}\in A)\geq 1/2$ follows from a symmetry argument as in the proof of (i).
\\Finally the proof of (iii) is analogous to the proof of (i). We shall show
that with probability 1, some sites become, and remain forever higher than their
neighbours. When this happens the configuration is broken in two disjoint parts,
but this time infinite boundaries can be created and have to be taken into
account. For this reason it is necessary to study the three types of boundary
conditions (0, 1 or 2 infinite boundaries) to make the induction work. Thus our
inductive hypothesis contains three statements. Let\\ $(\H_n)$: For any
$x\in\NN^n$, $t^{-1}X^n_t$ converges $\PP_x$-\as (resp. $\PP^{\infty}_x$-\as\@
and $\PP^{0,\infty}_x$-\as) to some random variable $G$ (resp. $G^\infty$ and
$G^{0,\infty}$), which takes the form 
\begin{equation}(b_\ell,\beta_0,a_1,\beta_0,a_2,\dots,a_{k-1},\beta_0,a_k,\beta_0,b_r),\label{forme2}\end{equation}
where $b_\ell$ and $b_r$ are given by:
\begin{itemize}
\item[-] for $G$, $b_\ell,b_r=(\beta_1)$ or $()$;
\item[-] for $G^{0,\infty}$, $b_\ell=(\beta_1)$ or $()$, and $b_r=(\beta_2)$ or $(v^{2,\infty},v^{2,\infty})$;
\item[-] for $G^\infty$, $b_\ell,b_r=(\beta_2)$ or $(v^{2,\infty},v^{2,\infty})$.
\end{itemize}
It is tedious but easy to check that vectors of the form (\ref{forme2})
concatenate together into a vector of $\A_3$. Since $(\H_1)$ and $(\H_2)$ are
straightforward, the problem is again reduced to showing that the separation in
two blocks occurs almost surely for $n\geq 3$. Putting $D_i=\{h:h(i-1)\geq
0,h(i)\leq 0\}$ for $i=1,\dots n$ (the signs of $h(0)$ and $h(n)$ are stressed
by the boundaries), we have to show that:
$$\PP_h(\cup_{i=1}^{n}F_{D_i})=1,\quad
\PP_h^\infty(\cup_{i=2}^{n1}F_{D_i})=1,\quad\text{ and
}\PP_h^{0,\infty}(\cup_{i=1}^{n-1}F_{D_i})=1.$$ But $\beta_0$ now is the
smallest parameter, so we easily get the analogous of (\ref{troustable}), and it
remains to prove that
$$\PP_h(T_D<+\infty)=1,\quad\PP^\infty_h(T_{D^\infty}<+\infty)=1,\quad\text{
and}\quad\PP^{0,\infty}_h(T_{D^{0,\infty}}<+\infty)=1,$$
where $D=\cup_{i=1}^{n}D_i$, $D^\infty=\cup_{i=2}^{n-1}D_i$ and
$D^{0,\infty}=\cup_{i=1}^{n-1}D_i$.
\\ The first equality is straightforward since any configuration belongs to the
set $D$. To prove the second and third equalities we can follow exactly the
proof of (i), except that $A_i$ is replaced by $D_i$, $\beta_2$ and $\beta_0$
invert their roles, $v^2$ is replaced by $v^{2,\infty}$ and the sets $C_i$ are
defined with opposite inequalities. 
\end{proof}

\section{Proof of Theorems \ref{recurrence0} and \ref{recurrence1} }\label{preuverecurrence}
We recall that in this section we always assume that $$\beta_0<\beta_2<\beta_1.$$
We first need to introduce some further notations: 
\begin{itemize}
\item[-] $\Delta_j^{s,t}X^n:=X^n_t(j)-X^n_s(j)$ ;
\item[-] $\tau^j_t:=\sup\{s\leq t: \Delta_j X^n_s=0\}$. This is not a stopping time.
\item[-] $P_\lambda$ will stand for a random variable with Poisson($\lambda$) distribution.
\item[-] $a^+:=\max(a,0)$, $a\in\RR$.
\end{itemize}
\textbf{Remark.} Since $\Delta_jX^n_t$ has the same distribution as
$-\Delta_{n-j}X^n_t$, showing the exponential tightness of $(H^n_t,t\geq 0)$
amounts to checking that for any $j$, $\PP_\ze(\Delta_j X^n_t\geq
k)=\oexpt{t}{k}$, that is exponential tightness for $((\Delta_j X^n_t)^+,t\geq
0)$.
\begin{Lem}
\label{CS}
We assume that for some $C_j,\alpha_j>0$, 
\begin{equation}\forall t\geq
0,~\forall\ell\in\NN\cap[0,t],~\PP_\ze(\Delta_jX^n_t>0;~\tau^j_t\in[\ell-1,\ell[
)\leq C_j e^{-\alpha_j(t-\ell)}.\label{eqCS}\end{equation} 
Then $((\Delta_j X^n_t)^+,t\geq 0)$ is exponentially tight.
\end{Lem}

\begin{proof}
Let $t\geq 0$, and put $m=\min\{q\in\NN:t-q\leq\frac{k}{2\beta_1}\}$. We have
$k/(4\beta_1)\leq(t-m)\leq k/(2\beta_1)$, as soon as $k\geq 4\beta_1$ and
$t\geq\frac{k}{2\beta_1}.$ But we may suppose these two restrictions fulfilled:
the first one because the conclusion does not depend on the values of
$\PP_\ze(\Delta_jX^n_t\geq k)$ for any finite number of $k$, and the second one
because, if it is not then the conclusion easily follows from
$\PP_\ze(\Delta_jX^n_t\geq k)\leq\PP(P_{\beta_1 t}\geq k)\leq\PP(P_{k/2}\geq
k)=\oexp{k}$.\\
 We decompose $$\PP_\ze(\Delta_jX^n_t\geq k)\leq\PP_\ze(\Delta_jX^n_t\geq
k,\tau^j_t\geq m)+\PP_\ze(\Delta_jX^n_t>0,\tau^j_t\leq m).$$ In this sum, the
first term is less than $\PP(N_{j,1}(t)-N_{j,1}(m)\geq
k)\leq(P_{\beta_1(t-m)}\geq k)\leq\PP(P_{k/2}\geq k)=\oexp{k}$, and the second
term is bounded by $\sum_{\ell\leq m}C_j
e^{-\alpha_j(t-\ell)}\leq\sum_{t-\ell\geq k/(4\beta_1)}C_j
e^{-\alpha_j(t-\ell)}\leq\sum_{u\geq 0}C_j
e^{-\alpha_j(k/(4\beta_1)+u)}=C_j(1-e^{-\alpha_j})^{-1}e^{-\frac{\alpha_j}{
\beta_1} k}.$
\end{proof}

\begin{Lem}
\label{heredite}
Let $\beta_0<\beta_2<\beta_1$ and $1\leq j\leq n-1$. We suppose that for
$i=1,\dots,j-1$, $((\Delta_i X^n_t)^+,t\geq 0)$ is exponentially tight, and that
there exists $d_j<\beta_2$ such that
\begin{equation}\label{sitebord}\PP_\ze(\tilde X^j_t(j)\geq d_j
t)=\oexp{t},\end{equation}
where $\tilde X^j$ is defined by (\ref{reftilde}) Then $((\Delta_j
X^n_t)^+,t\geq 0)$ is exponentially tight.
\\For $j=n-1$, it is not necessary to assume (\ref{sitebord}).
\end{Lem}

\begin{proof}
We first take $j\leq n-2$, and choose a constant $L>0$ such that 
\begin{equation}d_j+\frac jL<\beta_2.\label{defl}\end{equation}
The conclusion will follow from (\ref{eqCS}) that we now prove. The events
$$A_1:=\left\{\max_{i=1,\dots j-1}\Delta_iX^n_\ell>\tl\right\}\text{ and
}A_2:=\left\{\Delta_j X^n_\ell>\tl,\tau^j_t\in[\ell-1,\ell[\right\}$$ 
satisfy 
$$\PP_\ze(A_1),\PP_\ze(A_2)\leq D_j e^{-\gamma_j(t-\ell)},\text{ for some
}D_j,\gamma_j>0.$$
The bound for $\PP_\ze(A_1)$ holds by assumption, and the bound for
$\PP_\ze(A_2)$ holds because
$\PP_\ze(A_2)\leq\PP(N_{j,1}(\ell)-N_{j,1}
(\ell-1)>\tl)=\PP(P_{\beta_1}>\tl)$. We now remark that
$\{\Delta_jX^n_t>0;~\tau^j_t\in[\ell-1,\ell[\}\cap A_1^c\cap
A_2^c\subset\{\Delta_j X^n_s>0,\forall s\in[\ell,t];~\max_{i=1,\dots
j}\Delta_iX^n_\ell\leq\tl\}$, so it now remains to show that for some
$C_j,\alpha_j>0$,
$$\forall t\geq 0,\forall\ell\leq t,\PP_\ze\left(\Delta_j X^n_s>0,\forall
s\in[\ell,t];~\max_{i=1,\dots j}\Delta_iX^n_\ell\leq\tl\right)\leq C_j
e^{-\alpha_j(t-\ell)}.$$
Denoting by $A$ this last event, we note that
$A\subset\{\Delta_j^{\ell,t}X^n>(d_j+\frac{j-1}{L})(t-\ell)\}\cup\{\Delta_{j+1}^
{\ell,t}X^n\leq(d_j+\frac{j}{L})(t-\ell)\}$, so 
\begin{eqnarray}\label{somme}\PP_\ze(A)\leq
&\PP_\ze\left(A\cap\{\Delta_{j+1}^{\ell,t}X^n\leq(d_j+\frac{j}{L})(t-\ell)\}
\right) \nonumber \\&+
\PP_\ze\left(A\cap\{\Delta_j^{\ell,t}X^n>(d_j+\frac{j-1}{L})(t-\ell)\}\right).
\end{eqnarray}
Since $\{\Delta_{j+1}^{\ell,t}X^n\leq(d_j+\frac{j}{L})(t-\ell)\}\cap\{\Delta_j
X^n_s>0,\forall
s\in[\ell,t]\}\subset\{N_{j+1,2}(t)-N_{j+1,2}(\ell)\leq(d_j+\frac{j}{L}
)(t-\ell)\}$, the first term in the sum (\ref{somme}) is less than
$$\PP\left(P_{\beta_2(t-\ell)}\leq(d_j+\frac{j}{L})(t-\ell)\right),$$ which is
$\oexp{t-\ell}$ by (\ref{defl}). Putting $E_j=\{x\in\NN^j:\max_{i=1,\dots
j-1}\Delta_i x\leq\tl\}$ and using the Markov property, the second term in
(\ref{somme}) is less than

\begin{align*}
\sup_{x\in
E_j}\PP_x\left(\Delta_j^{0,t-\ell}X^j>(d_j+\frac{j-1}{L})(t-\ell)\right) &\leq
\sup_{x\in E_j}\PP_x\left(\Delta_j^{0,t-\ell}\tilde
X^j>(d_j+\frac{j-1}{L})(t-\ell)\right)\\ 
&\leq\PP_\ze\left(\tilde X^j_{t-\ell}(j)>d_j(t-\ell)\right)\\ &=\oexp{t-\ell}, 
\end{align*}

where the first inequality follows from Lemma \ref{lemcouplage}, the second one
from Lemma \ref{mon} and the fact that $\max(x)-x(j)\leq (j-1)/L$ for $x\in
E_j$, and the equality is assumption (\ref{sitebord}). This concludes the proof
for $j\leq n-2$.
\\We now treat the case $j=n-1$. Applying Theorem \ref{AMS} to $\tilde X^n$, we
get (\ref{sitebord}) for some constant $d_{n-1}<\beta_1$. This time we take $L$
such that $d_{n-1}+\frac {n-1}{L}<\beta_1$. We still have (\ref{somme}). Note
that on the event $\{\Delta_{n-1}X^n_t>0\}$, we must have $V_n(X^n_t)=1$. Hence
in the sum (\ref{somme}) we proceed as for $j<n-1$ for the second term, and the
first term is less than
$\PP\left(P_{\beta_1(t-\ell)}\leq\left(d_{n-1}+\frac{n-1}{L}
\right)(t-\ell)\right)=\oexp{t-\ell}$.
\end{proof}

\begin{proof}[Proof of Theorem \ref{recurrence0} ]
With (\ref{dtildencroissant}) in mind, our hypothesis implies that
$\beta_2>\d_k(\beta_1)$ for $k\leq n$. Hence we only have to prove the desired
result with $k=n+2$. Let us take $r\in]\d_n(\beta_1),\beta_2[$. By Lemma
\ref{lemcouplage}, we have $$\PP_\ze(\tilde X^j_t(j)\geq
rt)=\oexp{t},~j=1,\dots,n.$$
We show by induction on $j$ that for $j=1,\dots n+1$, we have:
$$(\H_j):\qquad((\Delta_j X^{n+2}_t)^+,t\geq 0)\text{ is exponentially tight,} $$
which by the remark preceding Lemma \ref{CS} is a sufficient condition for
$H^{n+2}$ to be ergodic. For $j=1$ we simply apply Lemma \ref{heredite}, whose
assumptions are clearly satisfied since $\beta_2>\beta_0$ and $\X^1_t$ is a
simple Poisson process with intensity $\beta_0$. For $j\leq n$, the fact that
$(\H_i),~i=1,\dots,j-1,$ imply $(\H_j)$ is a direct consequence of Lemma
\ref{heredite}. For $j=n+1$ it is still the case using the last assertion of
Lemma \ref{heredite}.

\end{proof}


\begin{proof}[Proof of Theorem \ref{recurrence1} ]
In the light of Theorem \ref{recurrence0}, we shall be able to conclude if we
show that for any $\eps>0$:
\begin{equation}
\PP_\ze\left(\X^n_t(n)\geq(n\beta_0+\eps)t\right)=\oexp{t}
\label{a}
\end{equation}

\begin{equation}
\PP_\ze\left(\X^n_t(n)\geq\left(\frac{(n-1)\beta_1+\beta_0}{n}
+\eps\right)t\right)=\oexp{t}
\label{b}
\end{equation}

\begin{equation}
\PP_\ze\left(\X^n_t(n)\geq(4\sqrt 2\sqrt{\beta_1\beta_0}+\eps)t\right)=\oexp{t}
\label{c}
\end{equation}

First (\ref{a}) simply follows from $\tilde
X^n_t(n)\leq\max_{i=1,\dots,n}\tilde X^n_t(i)$ and the fact that
$\max_{i=1,\dots,n}\tilde X^n_t(i)$ is dominated by a Poisson process with
intensity $n\beta_0$.
\\We now prove (\ref{b}). For notational convenience we define
$g_n:=n^{-1}((n-1)\beta_1+\beta_0)$. Let us take $\eta<2(n-1)^{-1}\eps $ and let
$\delta=n\eps-n(n-1)\eta/2>0$. We have 
\begin{align*}
\PP_\ze\left(\X^n_t(n)\geq\left(g_n+\eps\right)t\right)&\leq
\PP_\ze\left(\X^n_t(n)\geq\left(g_n+\eps\right)t;~\min_{i=1,\dots
n-1}\Delta_i\X^n_t\geq -\eta t\right)\\ 
&+\PP_\ze\left(\min_{i=1,\dots n-1}\Delta_i\X^n_t\leq -\eta t\right)
\end{align*}
For $x\in\NN^n$ we define $$\Sigma x=\sum_{i=1}^n x(i).$$ Then
\begin{align*}
\PP_\ze\Big(&\X^n_t(n)\geq\left(g_n+\eps\right)t;~\min_{1\leq i\leq
n-1}\Delta_i\X^n_t\geq -\eta t\Big)\\ &\leq \PP_\ze\Big(
\Sigma\X^n_t\geq\sum_{j=1}^n(g_n+\eps-(n-j)\eta)t \Big)\\ 
&\leq \PP_\ze\Big( \Sigma\X^n_t\geq(ng_n+\delta)t \Big)\\
&=\oexp{t},
\end{align*}
because in any configuration $x$, $V_j(x)=0$ for at least one site, hence
$\Sigma\X^n_t$ is dominated by a Poisson process with intensity $ng_n$. The fact
that also $$\PP_\ze\left(\min_{i=1,\dots n-1}\Delta_i\X^n_t\leq -\eta
t\right)=\oexp{t}$$ is a direct consequence of Theorem \ref{AMS}.
\\We finally turn to the proof of (\ref{c}), and let
$t_0:=1/(4\sqrt{2}\sqrt{\beta_1\beta_0})$. We shall show that for
$k\geq 1$, $j\geq 0$ and $1\leq i\leq n$, 
\begin{equation}\label{queue}
p^i_{k,j}:=\PP_\ze\left(\X^n_{kt_0}(i)\geq
k+j-1\right)\leq\left(\frac 12\right)^j.
\end{equation}
This implies the desired result: if (\ref{queue}) holds, then for $\eps>0$, 
$$\PP_\ze(\X^n_t(n)\geq(1/t_0+\eps)t)\leq\PP_\ze(\X^n_{t_0(\lfloor
t/t_0\rfloor+1)}(n)\geq\lfloor t/t_0\rfloor+\lfloor\eps t\rfloor)\leq
(1/2)^{\lfloor\eps t\rfloor}=\oexp{t}.$$ Here $\lfloor t/t_0\rfloor$ stands for
the integer part of $t/t_0$. To prove (\ref{queue}) we proceed by
induction, showing that $(\H_\ell)$ holds for any $\ell\geq 1$, where
$$(\H_\ell):\quad \forall k\geq 1,j\geq 0\text{ with }k+j=\ell,\forall
i\in\{1,\dots n\},~p^i_{k,j}\leq(1/2)^j.$$
In this proof we may and will suppose that 
\begin{equation}\beta_1\geq 2\beta_0,\label{wlg}\end{equation}
since otherwise we easily get $\d_n(\beta_1)\leq\beta_1\leq 2\beta_0\leq 4\sqrt
2\sqrt{\beta_1\beta_0}$. For readability we define
$\tau_{v,d}:=\inf\{s\geq 0:\X^n_s(v)=d\}$.\\
A site $i$ is said to be a \emph{seed} at level $\ell$ if the $\ell$-th square
to be deposed at site $i$ is added at a moment when site $i$ is at least as
high as its neighbours. This means that $$V_i(\tilde
X^n_{(\tau_{i,\ell})^-})=0.$$
For $i_1\leq i_2$ we say that $i_1$ \emph{extends} to $i_2$ during the time
interval
$[s,t]$ if $N_{1,i_1+1}$, $N_{1,i_1+2},$ $\dots$, $N_{1,i_2}$ jump successively
between times $s$ and $t$. For $i_1>i_2$ this definition is extended in an
obvious way.\\
Inequality (\ref{queue}) for $j=0$, and hence
$(\H_1)$, are straightforward. We now suppose that $(\H_\ell)$ holds and take
$i\in\{1,\dots n\}$ and $k,j\geq 1$ with $k+j=l+1$. Then
\begin{align*}
p^i_{k,j}&=\PP_\ze\Big(\X^n_{kt_0}(i)\geq
k+j-1\Big)\\
&\leq\sum_{u=1}^n\sum_{m=1}^k\PP_\ze\Big(\X^n_{(m-1)t_0}
(u)<k+j-1;~\X^n_{mt_0}(u)\geq\\  \\
 & \qquad\qquad k+j-1;~u\text{ is a seed at level } k+j-1;~u\text{ extends  }\\
  & \qquad\qquad\text{to } i\text{ during }[\tau_{u,k+j-1},kt_0]\Big).\\
\end{align*}
For this last event to be realized, the three following conditions have to be
satisfied:
\begin{itemize}
\item[-] $\tau_{u,k+j-2}\leq mt_0$,
\item[-] $N_{0,u}$ jumps at least one time in the time
interval
$[\max(\tau_{u,k+j-2},(m-1)t_0),mt_0]$,
\item[-] $u$ extends to $i$ after the first one of these jumps.
\end{itemize}
Using the fact that the Poisson distribution satisfies $\PP(P_\lambda\geq
1)\leq\lambda$, it follows that
\begin{equation*}
p^i_{k,j}\leq \sum_{u=1}^n\sum_{m=1}^k
p^u_{m,k-m+j-1}\beta_0t_0\PPP{P_{(k-m+1)\beta_1 t_0}\geq|u-i|},
\end{equation*}
and by the inductive hypothesis,
\begin{align*}
 p^i_{k,j}  & \leq \sum_{m=1}^k
(1/2)^{k-m+j-1}\beta_0t_0\sum_{u=1}^n\PPP{P_{(k-m+1)\beta_1 t_0}\geq|u-i|}
\\
& \leq \sum_{m=1}^k
(1/2)^{k-m+j-1}\beta_0t_0\sum_{v\in\ZZ}\PPP{P_{(k-m+1)\beta_1 t_0}\geq|v|}
\\
& \leq (1/2)^{j-1}\beta_0t_0\sum_{m=1}^k
(1/2)^{k-m}(1+2\EE[P_{(k-m+1)\beta_1 t_0}])
\\
& \leq (1/2)^{j-1}\beta_0t_0\left[\sum_{r=0}^{k-1}(1/2)^r+2\beta_1
t_0\sum_{r=0}^{k-1}(r+1)(1/2)^r\right] \\
& \leq (1/2)^{j-1}\beta_0t_0(2+8\beta_1t_0)\\
&\leq (1/2)^j.
\end{align*}
The last inequality is a consequence of (\ref{wlg}) and the choice of $t_0$.
\end{proof}

\section{Proof of Theorem \ref{transience}} 
The next lemma tells that when $n=3$, taking $\beta_1$ very large makes the
growth rate $v^3$ close to its maximum value $3\beta_0$. We recall that by
Corollary (\ref{n4}), $v^3$ exists as soon as we take $\beta_1,\beta_2>\beta_0$
\begin{Lem}\label{vitesse}
Let $\beta_1,\beta_2>\beta_0$ and $\eps>0$. We suppose that 
\begin{equation}\beta_1\geq 
\frac{27\beta_0^2\beta_2}{\eps(\beta_2-\beta_0)}.\label{hyplem}\end{equation}
Then the growth rate satisfies $$v^3\geq 3\beta_0-\eps.$$
\end{Lem}
\begin{proof}
Here we denote by $A$ the set of configurations with a hole: $$A=\{h\in\ZZ^2:
h(1)>0,h(2)<0\}.$$ For $x\in\NN^3$ we also say that $x\in A$ if the shape of $x$
belongs to $A$. We define a double sequence of stopping times by letting:\\\\
\begin{tabular}{ll}
$T_0=0,$ & 
\\$U_1=\inf\{t\geq 0: X^3_t\in A\}$, & $T_1=\inf\{t>U_1: X^3_t\notin A\},$
\\$U_{k+1}=\inf\{t\geq T_k: X^3_t\in A\}$, &  $T_{k+1}=\inf\{t>U_{k+1}:
X^3_t\notin A\},\text{ for }k\geq 2.$ 
\end{tabular}\\\\
We also define $$Y_n:=\Sigma X^3_{T_n}.$$ The desired result will follow if we
show that $\lim_{n\to\infty}Y_n/T_n\geq 9\beta_0-3\eps$. We first claim that the
sequence $$I_n:=Y_n-(9\beta_0-3\eps)T_n$$ is a submartingale. By the strong
Markov property this is the case if for any $x\notin A$,
\begin{equation}\EE_x[Y_1-(9\beta_0-3\eps)T_1]\geq
0\label{sousmart}.\end{equation} To establish this inequality we make the
following observations: 
\begin{itemize}
\item[-]Let $Z_1=\Sigma X^3_{U_1}$ be the number of jumps before hitting $A$.
Starting from any $y\notin A$ the probability that the first jump leads to $A$
is less than $\beta_0/\beta_1$. Hence by the Markov property, $Z_1$ is
stochastically larger than the geometrical distribution with parameter
$\beta_0/\beta_1$, and consequently
\begin{equation}\EE_x[Z_1]\geq\beta_1/\beta_0\label{obs1}.\end{equation}
\item[-]Any $y\notin A$ with $\Sigma y\notin 3\ZZ$ has at least one site $j$
such that $V_j(y)=1$. Thus conditionally on $Z_1$, at least $(2Z_1/3-1)$
transitions until time $U_1$ occur with a rate larger than $\beta_1$, and the
others occur with a rate at least $3\beta_0$. We deduce from this remark that
$\EE_x[U_1|Z_1]\leq\frac{2Z_1}{3\beta_1}+\frac{Z_1/3+1}{3\beta_0}$, and hence
\begin{equation}\EE_x[U_1]\leq
\left(\frac{2}{3\beta_1}+\frac{1}{9\beta_0}\right)
\EE_x[Z_1]+\frac{1}{3\beta_0}\label{obs2}.\end{equation}
\item[-]The configuration $X^3_{U_1}$ belongs to the set $A\cap\{x\in\NN^3:
x(1)=x(2)+1\text{ or }x(3)=x(2)+1\}$. But clearly, for any $y$ in that set, the
exit time from $A$ starting from $y$ is stochastically smaller than the hitting
time of $0$ for a birth and death process on $\ZZ_+$ starting from $1$, with
birth rate $\beta_0$ and death rate $\beta_2$. Hence
\begin{equation}
\EE_x[T_1-U_1]\leq\frac{1}{\beta_2-\beta_0}\label{obs3}.
\end{equation}

\end{itemize}
Now (\ref{obs2}) and (\ref{obs3}) yield
\begin{align*}
\EE_x[Y_1-(9\beta_0-3\eps)T_1] & \geq \EE_x[Z_1]-(9\beta_0-3\eps)\left[
\left(\frac{2}{3\beta_1}+\frac{1}{9\beta_0}\right)\EE_x[Z_1]+\frac{1}{3\beta_0}
+\frac{1}{\beta_2-\beta_0}\right] \\
 & =  \left[ \frac{2\eps-6\beta_0}{\beta_1}+ \frac{\eps}{3\beta_0} 
\right]\EE_x[Z_1]-(9\beta_0-3\eps)\left(\frac{1}{3\beta_0}+\frac{1}{
\beta_2-\beta_0}\right).
 \end{align*}
Condition (\ref{hyplem}) implies that $(2\eps-6\beta_0)/\beta_1+
\eps/(3\beta_0)\geq 0$. Using (\ref{obs1}) we then have
\begin{align*}
\EE_x[Y_1-(9\beta_0-3\eps)T_1] & \geq \left[ \frac{2\eps-6\beta_0}{\beta_1}+
\frac{\eps}{3\beta_0} 
\right]\frac{\beta_1}{\beta_0}-(9\beta_0-3\eps)\left(\frac{1}{3\beta_0}+\frac{1}
{\beta_2-\beta_0}\right) \\
 & \geq \eps \frac{\eps\beta_1}{3\beta_0^2}-9\frac{\beta_2}{\beta_2-\beta_0},
 \end{align*}
which is nonnegative under (\ref{hyplem}). Thus (\ref{sousmart}) holds and we
conclude that the sequence $I_n$ is a submartingale. 
\\ We define another sequence $(S_k,k\geq 0)$ of integers by letting $S_0=0$,
and for $k\geq 0$, $$S_{k+1}=\inf\{n>S_k:H^3_{T_n}=\ze\}.$$ We remark that
$T_{S_1}=\inf\{t\geq 0: H^3_t=\ze\text{ and }H^3_{t-}=(1,-1)\}$. But it is well
known that for any ergodic Markov process on a countable set and any two states
$s_1$ and $s_2$ with positive jump rate from $s_1$ to $s_2$, the time of first
transition from $s_1$ to $s_2$ has finite expectation. From this remark we
deduce that that $\EE_0[T_{S_1}]<\infty$, and consequently we also have
$\EE_0|I_{S_1}|<\infty$. The submartingale property gives us $\EE_0[I_{S_1}]\geq
0$. Since by the Markov property $I_{S_k}$ is the sum of $k$ independent copies
of variables distributed as $I_{S_1}$, and the same holds for $T_{S_k}$, an
application of the the law of large numbers gives:
$$3v^3=\lim_{k\to\infty}\frac{Y_{S_k}}{T_{S_k}}=\lim_{k\to\infty}\frac{I_{S_k}}{
T_{S_k}}+9\beta_0-3\eps=\frac{\EE_0[I_{S_1}]}{\EE_0[T_{S_1}]}+9\beta_0-3\eps\geq
9\beta_0-3\eps.$$ 
\end{proof}

\begin{proof}[Proof of Theorem \ref{transience}]
We first recall a basic fact. If $(\mathcal N_t,t\geq 0)$ is a Poisson process
with intensity $\lambda$, and $f$ is some nonnegative deterministic function
with $\lim_{\infty}f(t)/t=\mu>\lambda$, then 
\begin{equation}\PP(\forall t\geq 0,\mathcal N_t\leq
f(t))>0.\label{rappelpoisson}\end{equation}
From Proposition \ref{2sites} and Lemma \ref{vitesse} with
$\eps=3\beta_0-\beta_2$, we deduce that a sufficient condition for 
$$\beta_2<\min(v^2,v^3)$$ is that $\beta_1>B$, where 
$$B:=\max\left(\frac{\beta_0\beta_2}{2\beta_0-\beta_2},
\frac{27\beta_0^2\beta_2}{(3\beta_0-\beta_2)(\beta_2-\beta_0)}\right).$$ 
For any $n\geq 5$ it is possible to decompose the $n$ sites in blocks of length
$2$ or $3$ separated by holes of unit length. From now on we suppose for
notational convenience that $n\in 3\ZZ+2$, so we only use blocks of length $2$
and hence only $\beta_2<v^2$ is necessary. Of course this assumption could be
dropped and if it was, we would also need $\beta_2<v^3$. 
\\We start from configuration $x:=(1,1,0,1,1,0,\dots,1,1,0,1,1)$ and show that
the event $E=\{\forall t\geq 0,H^n_t\neq\ze\}$ satisfies $\PP_x(E)>0$. We use
the notation (\ref{defxab}) and remark that
\begin{align*}
E & \supset \Big\{\forall t\geq 0,
X^n_t(3)<\min\big(X^n_t(2),X^n_t(4)\big)\Big\}\cap\dots\cap\Big\{\forall t\geq
0, X^n_t(n-2)<\\ &\qquad \min\big(X^n_t(n-3),X^n_t(n-1)\big)\Big\} \\
 & = \Big\{\forall t\geq 0, N_{3,2}(t)<
\min\big(X_t^{1:2,(0,0)}(2),X_t^{4:5,(0,0)}(1)\big)\Big\}\cap\dots\cap\Big\{
\forall t\geq 0, N_{n-2,2}(t)<\\ &\qquad
\min\big(X_t^{n-4:n-3,(0,0)}(2),X_t^{n-1:n,(0,0)}(1)\big)\Big\}.  
\end{align*}
The process $m_t=\min\big(X_t^{1:2}(2),X_t^{4:5}(1),\dots,X_t^{n-4:n-3}(2),X_t^{n-1:n}(1)\big)$ satisfies $$\lim_{t\to\infty}\frac{m_t}{t}=v^2,$$ and is independent of the Poisson processes $N_{3,2},N_{6,2},\dots,N_{n-2,2}$. Hence the result follows from (\ref{rappelpoisson}) and the fact that $\beta_2<v_2$.
\end{proof}

\textbf{Acknowledgements.}\quad The author wishes to thank Enrique Andjel and
Étienne Pardoux for their continuous support during his research for this paper.
He also wishes to express his sincere gratitude to Alexandre Gaudilli\`ere for
stimulating discussions and his important contribution to Theorem
\ref{recurrence1}.

\end{document}